\documentclass{amsart}

\usepackage{amsfonts,mathtools,mathrsfs}
\usepackage{amscd}
\usepackage{amsmath,amsthm}
\newtheorem{thm}{Theorem}[section]

\newtheorem{cor}[thm]{Corollary}

\theoremstyle{definition}
\newtheorem{definition}[thm]{Definition}

\theoremstyle{remark}

\newtheorem{remark}[thm]{Remark}

\numberwithin{equation}{section}

\newcommand{\R}{\mathbb{R}}  % The real numbers.

\begin{document}

\title{Moduli of Continuity for Viscosity Solutions}

\author{Xiaolong Li}
\address{Department of Mathematics, University of California, San Diego, 
La Jolla, CA 92093}
\email{xil117@ucsd.edu}

\mbox{\ }\let\thefootnote\relax
\begin{abstract}
In this paper, we investigate the moduli of continuity for viscosity solutions of a wide class of nonsingular quasilinear evolution equations and also for the level set mean curvature flow, which is an example of singular degenerate equations. 
We prove that the modulus of continuity is a viscosity subsolution of some one dimensional equation. This work extends B. Andrews' recent result on moduli of continuity for smooth spatially periodic solutions. 
\end{abstract}

 \maketitle
%\tableofcontents

\section{Introduction}
Given a function $u:\R^n \to \R$, any function $f:[0,\infty)\to \R_{+}$ satisfying 
$$|u(y)-u(x)| \leq 2 f\left(\frac{|y-x|}{2}\right)$$
for all $x$ and $y$ is called a modulus of continuity of $u$. 
The (optimal) modulus of continuity $\omega$ of $u$ is defined by 
$$\omega(s)=\sup \left\lbrace \frac{u(y)-u(x)}{2}\Big| \frac{|y-x|}{2}=s \right\rbrace.$$ 
The estimate of modulus of continuity has been studied by B. Andrews and J. Clutterbuck in several papers \cite{AC1} \cite{AC2}. 
B. Andrews and J. Clutterbuck \cite{AC3}, B. Andrews and L. Ni \cite{AN} and L. Ni \cite{N1} have also studied the modulus of continuity for heat equations on manifolds. 
%The following theorem appears in a recent survey \cite{A} by B. Andrews. 

%\begin{thm}
%Let $u:\R^n \times [0, T) \to \R$ be a spatially periodic smooth solution of the following isotropic flow, 
%\begin{equation}\label{isotropic flow}
%\frac{\partial u}{\partial t} = \left[a(|Du|) \frac{D_iu D_ju}{|Du|^2}+b(|Du|)\left(\delta_{ij}-\frac{D_iu %D_ju}{|Du|^2}\right)\right]D_iD_ju, 
%\end{equation} 
%then its modulus of continuity is a viscosity subsolution of the corresponding one-dimensional heat equation 
%$\omega_t=a(\omega^{\prime})\omega^{\prime \prime}$ on $(0,\infty)\times (0,T)$. 
%\end{thm}
More precisely, B. Andrews and J. Clutterbuck considered the following quasilinear evolution equation
\begin{equation}\label{general quasilinear equation}
u_t=a^{ij}(Du, t) D_iD_j u+b(Du, t) 
\end{equation}
where $A(p,t)=\left(a^{ij}(p, t)\right)$ is positive semi-definite. Under the assumption that there exists a continuous function 
$\alpha: \R_{+}\times [0, T] \to \R$ with
\begin{equation} \label{assumption equation}
0 < \alpha(R, t) \leq R^2 \inf_{|p|=R, (v \cdot p) \neq 0} \frac{v^{T}A(p, t)v}{(v \cdot p)^2},
\end{equation}
They showed \cite[Theorem 3.1]{AC2} that the modulus of continuity of a regular periodic solution to \eqref{general quasilinear equation} is a viscosity subsolution of the one dimensional equation 
\begin{equation} \label{general one dim equaiton}
\phi_t = \alpha(|\phi^{\prime}|, t) \phi^{\prime \prime}.
\end{equation}
Note that their result is applicable to any anisotropic mean curvature flow and can be used to obtain gradient estimate and thus existence and uniqueness of \eqref{general quasilinear equation}. 
%has the form of equation \eqref{general quasilinear equation} and satisfies assumption \eqref{assumption equation}. 

The first result of this paper is that the same holds for viscosity solutions of \eqref{general quasilinear equation} 
when \eqref{assumption equation} holds and $a^{ij}, b :\R^n \times [0, T] \to \R$ are continuous functions.  

\begin{thm}\label{thm for general periodic equation}
Let $u: \R^n \times [0, T] \to \R$ be a continuous periodic viscosity solution of \eqref{general quasilinear equation}. Then the modulus of continuity $\omega(s,t) =\sup \left\lbrace \frac{u(y,t)-u(x,t)}{2}\Big| \frac{|y-x|}{2}=s \right\rbrace$ of $u$ is a viscosity subsolution of the one dimensional equation \eqref{general one dim equaiton}.
\end{thm}

We also study the modulus of continuity for singular evolution equations. 
As summarized in a recent survey \cite{A} by B. Andrews, for the isotropic flows of the form
\begin{equation}\label{isotropic flow}
u_t = \left[a(|Du|) \frac{D_iu D_ju}{|Du|^2}+b(|Du|)\left(\delta_{ij}-\frac{D_iu D_ju}{|Du|^2}\right)\right]D_iD_ju, 
\end{equation} 
the modulus of continuity of a spatially periodic smooth solution of \eqref{isotropic flow} is a viscosity subsolution of the corresponding one-dimensional heat equation $\omega_t=a(\omega^{\prime})\omega^{\prime \prime}$. Note that equation \eqref{isotropic flow} covers the classical heat equation, the graphical mean curvature flow and the $p$-Laplace heat equation with suitable choices of $a$ and $b$. When \eqref{isotropic flow} is nonsingular, it is covered by \eqref{general quasilinear equation}. When it is singular, it has to be treated differently since there are various definitions for viscosity solutions of singular equations. We will focus on the particular case $a=0$ and $b=1$, which corresponds to the level set mean curvature flow:
\begin{equation}\label{level set mean curvature flow}
u_t= \left(\delta_{ij}-\frac{D_iu D_ju}{|Du|^2}\right)D_iD_ju.
\end{equation} 
Equation \eqref{level set mean curvature flow} was studied by L. Evans and J. Spruck in \cite{ES1}. They gave a definition of viscosity solution and proved that for an initial data $g$ that is continuous and constant on $\mathbb{R}^n \cap \{|x|\geq S\}$, there exists a unique viscosity solution $u$ that is continuous and constant on $\mathbb{R}^n \cap \{|x|\geq R\}$, with $R$ depending only on $S$. We will recall their definition in Section 2. 

We prove the following theorem:
\begin{thm} \label{thm for level set mean curvature flow}
Let $u:\mathbb{R}^n \times [0,\infty) \to \mathbb{R} $ be a viscosity solution of \eqref{level set mean curvature flow} 
with continuous initial data $g$ that is a constant on $\mathbb{R}^n \cap \{|x|\geq S\}$. Then the modulus of continuity $\omega(s,t) =\sup \left\lbrace \frac{u(y,t)-u(x,t)}{2}\Big| \frac{|y-x|}{2}=s \right\rbrace$ of $u$ is a viscosity subsolution of $\omega_t=\max\{0, \frac 1 4 \left( \omega^{\prime \prime}+|\omega^{\prime \prime}|\right)\}$ on $(0,\infty) \times (0,\infty)$.
\end{thm} 
As an immediate consequence, we have that any concave modulus of continuity for the initial data is preserved by the level set mean curvature flow. 
\begin{cor} \label{cor}
Let $u:\mathbb{R}^n \times [0,\infty) \to \mathbb{R} $ be a viscosity solution of \eqref{level set mean curvature flow} 
with continuous initial data $g$ that is a constant on $\mathbb{R}^n \cap \{|x|\geq S\}$. 
Assume $\phi$ is nonnegative, concave and satisfies
$|g(y)-g(x)| \leq 2 \phi\left(\frac{|y-x|}{2}\right)$ for all $x, y$, then 
$$|u(y,t)-u(x,t)| \leq 2 \phi\left(\frac{|y-x|}{2}\right)$$
for all $x$, $y$ and $t\geq 0$. 
%In particular, if $\phi(x,0)$ is concave, then 
%$$|u(y,t)-u(x,t)| \leq 2 \phi\left(\frac{|y-x|}{2}, 0\right)$$
%for all $x$, $y$ and $t\geq 0$. 
\end{cor}
\begin{proof}[Proof of Corollary \ref{cor}]
Since the function $\phi_{\epsilon}=\phi + \epsilon e^t$ satisfies 
$$\partial_t \phi_{\epsilon} >0=\max\{0, \frac 1 4 \left( \phi^{\prime \prime}+|\phi^{\prime \prime}|\right)\},$$
so it cannot touch $\omega$ from above by Theorem \ref{thm for level set mean curvature flow}. 
\end{proof}

\section{Definitions of Viscosity Solutions}

We give definition of a viscosity solution for the general equation
\begin{equation} \label{nonsingular equation}
u_t+F(x, t, u, \nabla u, \nabla^2 u)=0
\end{equation} 
assuming $F: \R^n \times [0,T] \times \R \times \R{^n}\times S^{n\times n} \to \R$ is continuous and degenerate elliptic.
Let $O$ be an open subset of $\Omega \times (0,T)$. We write $z=(x,t)$ and $z_0=(x_0, t_0)$. 

The following notations are useful:
$$\mbox{USC}(O)=\{u:O \to \R |\  u \mbox{  is upper semicontinuous  }\},$$
$$\mbox{LSC}(O)=\{u:O \to \R |\  u \mbox{  is lower semicontinuous  }\},$$
\begin{definition}
(i) A function $u \in \mbox{USC}(O)$ is a viscosity subsolution of \eqref{nonsingular equation}
%$$u_t+F(z,u,\nabla u, \nabla^2 u)=0$$ 
in $O$ if for any $\phi \in C^{\infty}(O)$ such that $u-\phi$ has a local maximum at $z_0 \in O$, then
\begin{align*}
\phi_t(z_0) + F(z_0,u(z_0),\nabla \phi(z_0), \nabla^2 \phi(z_0) ) \leq 0. 
\end{align*}

(ii) A function $u \in \mbox{LSC}(O)$ is a viscosity supersolution of \eqref{nonsingular equation}
%$$u_t+F(z,u,\nabla u, \nabla^2 u)=0$$ 
in $O$ if for any $\phi \in C^{\infty}(O)$ such that $u-\phi$ has a local minimum at $z_0 \in O$, then
\begin{align*}
\phi_t(z_0) + F(z_0,u(z_0),\nabla \phi(z_0), \nabla^2 \phi(z_0)) & \geq 0.
\end{align*}

(iii) A viscosity solution of \eqref{nonsingular equation} in $O$ is defined to be a continuous function that is both a 
viscosity subsolution and a viscosity supersolution of \eqref{nonsingular equation} in $O$. 
\end{definition}  

We have an equivalent definition in terms of parabolic semijets. Assume $u \in \mbox{USC}(O)$ and $z_0\in O$. 
The parabolic superjet of $u$ at $z_0$, denoted by $\mathcal{P}^{2,+}u(z_0) $, is defined by 
\begin{align*}
\mathcal{P}^{2,+}u(z_0)=&\{(\tau,p,X)\in \R \times \R^n \times S^{n\times n} | u(z) \leq  u(z_0) +\tau(t-t_0)+p\cdot(x-x_0)\\
& +\frac{1}{2} (x-x_0)^{T}X(x-x_0)+o(|x-x_0|^2+|t-t_0|) \mbox{  as  } z\to z_0 \}.
\end{align*}
The parabolic subjet of $u \in \mbox{LSC}(O)$ at $z_0$, denoted by $\mathcal{P}^{2,-}u(z_0) $, is defined by 
\begin{equation*}
\mathcal{P}^{2,-}u(z_0)=-\mathcal{P}^{2,+}(-u)(z_0).
\end{equation*}
%\begin{align}
%\mathcal{P}^{2,+}_{O}(z_0)=&\{(\tau,p,X)\in \R \times \R^n \times S^n | u(z) \leq  u(z_0) +\tau(t-t_0)+p\cdot(x-x_0)\\
%& +\frac{1}{2} (x-x_0)^{T}X(x-x_0)+o(|z-z_0|^2) \mbox{  as  } z\to z_0 \}.
%\end{align}

\begin{definition}
(i) A function $u \in \mbox{USC}(O)$ is a viscosity subsolution of \eqref{nonsingular equation}
%$$u_t+F(z,u,\nabla u, \nabla^2 u)=0$$ 
in $O$ if for all $(x,t) \in O$ and $(\tau, p, X) \in \mathcal{P}^{2,+}u(x,t)$, 
\begin{align*}
\tau +F(z, u(z), p, X) \leq 0.
\end{align*}

(ii) A function $u \in \mbox{LSC}(O)$ is a viscosity supersolution of \eqref{nonsingular equation}
%$$u_t+F(z,u,\nabla u, \nabla^2 u)=0$$ 
in $O$ if for all $(x,t) \in O$ and $(\tau, p, X) \in \mathcal{P}^{2,-}u(x,t)$, 
\begin{align*}
\tau +F(z, u(z) , p, X) \geq 0.
\end{align*}
\end{definition}

\begin{remark}
In the above definitions, since $F$ is continuous, we can replace $\mathcal{P}^{2,+}u(z_0)$ and $\mathcal{P}^{2,-}u(z_0)$ by 
$\overline{\mathcal{P}}^{2,+}u(z_0)$ and $\overline{\mathcal{P}}^{2,-}u(z_0)$ respectively, where 
the closures are defined by
\begin{align*}
\overline{\mathcal{P}}^{2,+}u(z_0)
&=\{(\tau,p,X)\in \R \times \R^n \times S^{n\times n} | 
\mbox{  there is a sequence  } (z_j,\tau_j,p_j,X_j) \\
&\mbox{  such that  } (\tau_j,p_j ,X_j)\in \mathcal{P}^{2,+}u(z_j) \\
&\mbox{  and  } (z_j,u(z_j),\tau_j,p_j,X_j) \to (z_0,u(z_0),\tau, p ,X) \mbox{  as  } j\to \infty \}.  \\
\overline{\mathcal{P}}^{2,-}u(z_0)&=-\overline{\mathcal{P}}^{2,+}(-u)(z_0). 
\end{align*}
\end{remark}

For singular equations, there are different ways to define viscosity solutions. 
For the level set mean curvature flow \eqref{level set mean curvature flow}, we use the definition given by L. Evans and J. Spruck in \cite{ES1}, where viscosity solutions are called weak solutions. 

\begin{definition} \label{def subsolution}
A continuous and bounded function $u:\R^n \times [0,\infty) \to \R$ is a viscosity subsolution of \eqref{level set mean curvature flow} if for any $\phi \in C^{\infty}(\R^{n+1})$ such that $u-\phi$ has a local maximum at a point $(x_0,t_0)\in \R^n \times (0, \infty)$, 
then we have
\begin{equation*}
\begin{cases}
   \phi_t \leq \left( \delta_{ij}-\frac{D_i\phi D_j\phi}{|D\phi|^2}\right) D_iD_j \phi \mbox{  at  } (x_0,t_0) \\
   \mbox{if  } D\phi(x_0,t_0) \neq 0,
\end{cases}\\
\end{equation*}
and
\begin{equation*}
\begin{cases}
\phi_t \leq \left( \delta_{ij}-\eta_i \eta_j\right) D_iD_j \phi \mbox{  at  } (x_0,t_0) \\
\mbox{  for some  } \eta \in \R^n \mbox{  with  } |\eta|\leq 1, \mbox{  if  } D\phi(x_0,t_0)=0
\end{cases}
\end{equation*}
\end{definition}

\begin{definition} \label{def supersolution}
A continuous and bounded function $u:\R^n \times [0,\infty) \to \R$ is a viscosity supersolution of \eqref{level set mean curvature flow} if for any $\phi \in C^{\infty}(\R^{n+1})$ such that $u-\phi$ has a local minimum at a point $(x_0,t_0)\in \R^n \times (0, \infty)$, 
then we have
\begin{equation*}
\begin{cases}
   \phi_t \geq \left( \delta_{ij}-\frac{D_i\phi D_j\phi}{|D\phi|^2}\right) D_iD_j \phi \mbox{  at  } (x_0,t_0) \\
   \mbox{if  } D\phi(x_0,t_0) \neq 0,
\end{cases}\\
\end{equation*}
and
\begin{equation*}
\begin{cases}
\phi_t \geq \left( \delta_{ij}-\eta_i \eta_j\right) D_iD_j \phi \mbox{  at  } (x_0,t_0) \\
\mbox{  for some  } \eta \in \R^n \mbox{  with  } |\eta|\leq 1, \mbox{  if  } D\phi(x_0,t_0)=0
\end{cases}
\end{equation*}
\end{definition}

\begin{definition} 
A continuous and bounded function $u:\R^n \times [0,\infty) \to \R$ is a viscosity solution of \eqref{level set mean curvature flow} 
provided $u$ is both a viscosity subsolution and a viscosity supersolution.  
\end{definition}

We also have alternative definitions in terms of parabolic semijets. 
\begin{definition}\label{alt def subsolution}
A continuous and bounded function $u:\R^n \times [0,\infty) \to \R$ is a viscosity subsolution of \eqref{level set mean curvature flow}
if for all $(x, t)\in \R^n \times (0, \infty)$ and $(\tau, p, X) \in \mathcal{P}^{2,+}u(x,t)$,  
\begin{equation*}
\tau \leq \left(\delta_{ij}-\frac{p_ip_j}{|p|^2}\right)X_{ij} \mbox{  if  } p\neq 0
\end{equation*}
and
\begin{equation*}
\tau \leq (\delta_{ij}-\eta_i\eta_j)X_{ij} \mbox{  for some  } |\eta|\leq 1, \mbox{  if  } p= 0.
\end{equation*}
\end{definition}

\begin{definition}\label{alt def supersolution}
A continuous and bounded function $u:\R^n \times [0,\infty) \to \R$ is a viscosity supersolution of \eqref{level set mean curvature flow} 
if for all $(x, t)\in \R^n \times (0, \infty)$ and $(\tau, p, X) \in \mathcal{P}^{2,-}u(x,t)$,  
\begin{equation*}
\tau \geq \left(\delta_{ij}-\frac{p_ip_j}{|p|^2}\right)X_{ij} \mbox{  if  } p\neq 0
\end{equation*}
and
\begin{equation*}
\tau \geq (\delta_{ij}-\eta_i\eta_j)X_{ij} \mbox{  for some  } |\eta|\leq 1, \mbox{  if  } p= 0.
\end{equation*}
\end{definition}

\begin{remark}
One can replace $\mathcal{P}^{2,+}u(z_0)$ and $\mathcal{P}^{2,-}u(z_0)$ by 
$\overline{\mathcal{P}}^{2,+}u(z_0)$ and $\overline{\mathcal{P}}^{2,-}u(z_0)$ respectively in the above definitions for the reason of continuity. 
\end{remark}

\section{Proof of Theorem \ref{thm for general periodic equation}}

\begin{proof}[Proof of Theorem \ref{thm for general periodic equation}]
We must show that if $\phi$ is a smooth function such that $\omega -\phi$ has a local maximum at $(s_0,t_0)$ 
 for $s_0>0$ and $t_0>0$, then  at $(s_0,t_0)$ 
 $$\phi_t \leq \alpha(|\phi^{\prime}|, t) \phi^{\prime \prime}.$$
 %Replacing $\phi$ by 
 %$$\varphi(s,t)=\phi(s,t)+(s-s_0)^4+(t-t_0)^4$$
 %if necessary, we may assume $\omega -\phi$ has a strict local maximum at $(s_0,t_0)$ 
 % for $s_0>0$ and $t_0>0$. 
 Since $u$ is continuous and periodic, there exist points $x_0$ and $y_0$ with $|y_0-x_0|=2s_0$ attaining the supremum, 
 $$\omega(s_0,t_0)=\frac{u(y_0,t_0)-u(x_0,t_0)}{2}.$$
 Define 
 \begin{equation*}
 Z(x,y,t):=u(y,t)-u(x,t)-2 \phi\left(\frac{|y-x|}{2},t\right). 
 \end{equation*}
 In view of the definition of $\omega$, we obtain that 
 $$Z(x,y,t)\leq Z(x_0,y_0,t_0)$$ 
 for all $|y-x|$ close to $2s_0$ and $t$ close to $t_0$. 
 Thus $Z$ has a local maximum at $(x_0,y_0,t_0)$. 
 Since $Z$ is continuous, 
 by the parabolic version maximum principle for semicontinuous functions \cite[Theorem 8.3]{CIL}, for any $\lambda >0$,  
 there exist $X, Y \in S^{n\times n}$ such that 
  \begin{equation*}\label{sub semijet}
   (b_1, 2 D_y\phi(s_0, t_0), X) \in \overline{\mathcal{P}}^{2,+} u(y_0,t_0),
  \end{equation*}
 \begin{equation*}\label{super semijet}
 (-b_2, -2D_x\phi(s_0, t_0), Y) \in \overline{\mathcal{P}}^{2,-} u(x_0,t_0),
 \end{equation*}
 \begin{equation*}\label{time semijet }
 b_1+b_2=2\phi_t(s_0, t_0),
 \end{equation*}
 
  \begin{equation}\label{Hessian inequality for quasilinear}
  -\left(\lambda^{-1}+\left\|M\right\| \right)I \leq
    \begin{pmatrix}
   X & 0 \\
   0 & -Y 
   \end{pmatrix}
   \leq M+\lambda M^2, 
  \end{equation}
 where $$M=2D^2\phi=2\begin{pmatrix*}
    D^2_y \phi & D^2_{y,x} \phi \\
    D^2_{x,y} \phi & D^2_x \phi
    \end{pmatrix*} 
    =\begin{pmatrix*}
       B  & -B  \\
       -B  & B 
       \end{pmatrix*} ,$$ 
   with $B=2D^2_y \phi(s_0, t_0)$. 
   
   To simplify, we choose an orthonormal basis of $\R^n$ with $e_n= \frac{y-x}{|y-x|}$, then
      $$2 D_y\phi(s_0, t_0) =-2 D_x\phi(s_0, t_0) =\phi^{\prime}(s_0, t_0) e_n.$$
      $$B= \begin{pmatrix}
           \frac{\phi^{\prime}}{2s_0} &  &  &  \\
            &  \ddots &  &  \\
            & &  \frac{\phi^{\prime}}{2s_0} &  \\
             &  &  & \frac 1 2 \phi^{\prime \prime}
          \end{pmatrix} .$$

   Since $u$ is both a subsolution and a supersolution of \eqref{general quasilinear equation},
    we have 
    \begin{equation*}
    b_1 \leq tr (A(\phi^{\prime} e_n) X) +b(\phi^{\prime} e_n)
    \end{equation*}
    \begin{equation*}
     -b_2 \geq tr (A(\phi^{\prime}e_n) Y) +b(\phi^{\prime} e_n)
     \end{equation*}
    By choosing a symmetric matrix $C$ such that $\begin{pmatrix}
       A & C \\
       C& A
       \end{pmatrix} \geq 0$, we obtain using \eqref{Hessian inequality for quasilinear}
   \begin{align*}
   2\phi_t(s_0,t_0) & =b_1+b_2  \leq tr \left(A(\phi^{\prime}e_n) (X-Y)\right) 
    =tr 
   \begin{pmatrix}
   A & C \\
   C& A
   \end{pmatrix} 
   \begin{pmatrix}
      X & 0 \\
      0 & -Y
   \end{pmatrix}  \\
   & \leq tr \begin{pmatrix}
      A & C \\
      C & A
      \end{pmatrix} 
      \begin{pmatrix}
         B & -B \\
        -B & B
      \end{pmatrix} + \lambda tr \begin{pmatrix}
            A & C \\
            C^{T}& A
            \end{pmatrix} 
            \begin{pmatrix}
               B & -B \\
              -B & B
            \end{pmatrix}^2 \\
     & = 2 tr \left((A-C)B\right) +4 \lambda tr \left((A-C)B^2\right)
     \end{align*}
   Taking $C=A-2 \alpha(|\phi^{\prime}|)e_n \otimes e_n$, it's easy to verify $\begin{pmatrix}
          A & C \\
          C& A
          \end{pmatrix} \geq 0$ due to \eqref{assumption equation}. 
    Thus at $(s_0, t_0)$
    \begin{equation}
    \phi_t \leq \alpha(|\phi^{\prime}|) \phi^{\prime \prime} + \lambda \alpha(|\phi^{\prime}|) (\phi^{\prime \prime})^2
    \end{equation}
Since $\lambda >0$ is arbitrary,  we get 
$$\phi_t \leq \alpha(|\phi^{\prime}|) \phi^{\prime \prime}$$ 
at $(s_0, t_0)$. 
\end{proof}

\section{Proof of Theorem \ref{thm for level set mean curvature flow}}

\begin{proof}[Proof of Theorem \ref{thm for level set mean curvature flow}]
Suppose that $\phi$ is a smooth function such that $\omega - \phi$ has a strict local maximum at $(s_0, t_0)$ with $s_0 >0$ and $t_0>0$. As in the proof of Theorem \ref{thm for general periodic equation}, we arrive at that the function 
\begin{equation*}
 Z(x,y,t):=u(y,t)-u(x,t)-2 \phi\left(\frac{|y-x|}{2},t\right). 
 \end{equation*}
 has a local maximum at $(x_0,y_0,t_0)$. Again we use an orthonormal frame with $e_n=\frac{y-x}{|y-x|}$. 
 The maximum principle for semicontinuous functions \cite[Theorem 8.3]{CIL} implies that for any $\lambda >0$,  
 there exist $X, Y \in S^{n\times n}$ such that 
  \begin{equation*}
   (b_1, \phi^{\prime}(s_0, t_0)e_n, X) \in \overline{\mathcal{P}}^{2,+} u(y_0,t_0)
  \end{equation*}
 \begin{equation*}
 (-b_2, \phi^{\prime}(s_0, t_0)e_n, Y) \in \overline{\mathcal{P}}^{2,-} u(x_0,t_0)
 \end{equation*}
 \begin{equation*}
 b_1+b_2=2\phi_t(s_0, t_0)
 \end{equation*}
 \begin{equation}\label{Hessian inequality}
  -\left(\lambda^{-1}+\left\|M\right\| \right)I \leq
  \begin{pmatrix}
 X & 0 \\
 0 & -Y 
 \end{pmatrix}
 \leq M+\lambda M^2,
 \end{equation}
 where $$M
    =\begin{pmatrix*}
       B  & -B  \\
       -B  & B 
       \end{pmatrix*} ,$$ 
   with $$B=2D^2_y \phi(s_0,t_0)= \begin{pmatrix}
              \frac{\phi^{\prime}}{2s_0} &  &  &  \\
               &  \ddots &  &  \\
               & &  \frac{\phi^{\prime}}{2s_0} &  \\
                &  &  & \frac 1 2 \phi^{\prime \prime}
             \end{pmatrix} .$$
For any vector $p \in \R^n$, we have 
 \begin{equation*}
 p^{T}Xp -p^{T}Yp = 
 (p,p)^{T}\begin{pmatrix}
  X & 0 \\
  0 & -Y 
  \end{pmatrix}(p,p) \leq (p,p)^{T}(M+\lambda M^2)(p,p) =0
 \end{equation*}
 Therefore $X \leq Y$. 
 For simiplicity, we denote $A(p)=I-\frac{p\otimes p}{|p|^2}.$  
If $\phi^{\prime}(s_0, t_0)\neq 0$, then by the definition of viscosity solution of \eqref{level set mean curvature flow},
\begin{align*}
b_1 \leq tr (A(\phi^{\prime}e_n) X), \\
-b_2 \geq tr (A(\phi^{\prime}e_n) Y). 
\end{align*}
Using the fact $X\leq Y$, we get 
$$\phi_t = \frac 1 2 (b_1+b_2) \leq \frac 1 2 tr\left(A(\phi^{\prime}e_n)(X-Y)\right) \leq 0.$$  

If $\phi^{\prime}(s_0, t_0)=0$, then it follows from the definition of a viscosity solutions that for some $\xi, \eta$ with $|\xi|, |\eta| \leq 1$, 
\begin{align*}
b_1 \leq tr (A(\xi) X), \\
-b_2 \geq tr (A(\eta) Y). 
\end{align*}
In view of \eqref{Hessian inequality}, we have $X\leq B+2\lambda B^2$ and $-Y \leq B+2\lambda B^2$. Thus 
$$tr((A(\xi) X))\leq tr\left((A(\xi) (B+2\lambda B^2)\right) = \frac 1 2(1-\xi_n^2) \left((\phi^{\prime \prime}+\lambda (\phi^{\prime \prime})^2 \right),$$
$$tr((A(\eta) Y))\geq tr\left((A(\eta) (-B-2\lambda B^2)\right) = -\frac 1 2(1-\eta_n^2) \left((\phi^{\prime \prime}+\lambda (\phi^{\prime \prime})^2 \right).$$
Therefore,
\begin{align*}
2\phi_t =b_1+b_2\leq tr (A(\xi) X) -tr (A(\eta) Y) \leq \left(1-\frac{\xi_n^2+\eta_n^2}{2}\right) \left(\phi^{\prime \prime} +\lambda (\phi^{\prime \prime})^2\right).
\end{align*}
%\begin{cases}
%0  & \mbox{ if } \phi^{\prime \prime} \leq 0, \\  
%\phi^{\prime \prime} & \mbox{ if } \phi^{\prime \prime} \geq 0.
%\end{cases}
Letting $\lambda \to 0$ yields 
$$\phi_t \leq \frac 1 4 \left(\phi^{\prime \prime}+|\phi^{\prime \prime}|\right).$$
\end{proof}

\section{Acknowledgement}
The author would like to thank his advisors Professor Lei Ni and Professor Ben Chow for their guidance and encouragement. 
%The author also benefits from many conversations with Paul Bryan, Zhenyang Li, Huaqiao Liu, Kui Wang and Yuntao Zhang.   
\bibliographystyle{plain}

\bibliography{myref}

\end{document}